\documentclass[11pt,letterpaper,colorlinks=true,linkcolor=blue,citecolor=blue,urlcolor=blue]{scrartcl}
\usepackage[T1]{fontenc}
\usepackage[utf8]{inputenc}
\usepackage{lmodern}
\usepackage[english]{babel}
\usepackage[english]{isodate}
\usepackage{amsthm}
\usepackage{amsmath}
\usepackage{amssymb}
\usepackage{amsfonts}
\usepackage{mathtools}
\usepackage{mathrsfs}
\usepackage{dsfont}
\usepackage{upgreek}
\usepackage{xpatch}
\usepackage{bm}
\usepackage{bbm}
\usepackage{enumitem}
\usepackage{float}
\usepackage{footnotebackref}
\usepackage[totalwidth=480pt, totalheight=680pt]{geometry}
\usepackage{hyperref}
\usepackage{tikz-cd}

\theoremstyle{definition}
\newtheorem{definition}{Definition}[section]

\theoremstyle{plain}
\newtheorem{theorem}[definition]{Theorem}
\newtheorem{proposition}[definition]{Proposition}
\newtheorem{lemma}[definition]{Lemma}

\theoremstyle{remark}
\newtheorem{remark}[definition]{Remark}

\addtokomafont{disposition}{\normalfont\scshape}

\makeatletter
\xpatchcmd{\@sec@pppage}{
\bfseries}{
\normalfont\scshape\Large}{}{}
\makeatother

\setkomafont{section}{\Large}
\setkomafont{subsection}{\large}

\numberwithin{equation}{section}

\newcommand{\PP}{\mathscr{P}}
\newcommand{\Rd}{\mathds{R}^{d}}
\newcommand{\lc}{\preceq_{\textnormal{c}}}
\newcommand{\R}{\mathds{R}}
\newcommand{\MCov}{\textnormal{MCov}}
\newcommand{\MT}{\mathsf{MT}}
\newcommand{\Cpl}{\mathsf{Cpl}}
\newcommand{\bary}{\operatorname{bary}}
\newcommand{\E}{\mathds{E}}
\newcommand{\Law}{\operatorname{Law}}
\newcommand{\dom}{\operatorname{dom}}
\newcommand{\Cqs}{C_{2}(\mathds{R}^{d})}
\newcommand{\Cbqs}{C_{\textnormal{b},2}(\mathds{R}^{d})}
\newcommand{\conv}{\operatorname{conv}}

\begin{document}


\title{
\LARGE \texorpdfstring{$q$}{q}-Bass martingales\thanks{
I am grateful to Julio Backhoff-Veraguas, Mathias Beiglb\"ock and Walter Schachermayer for their useful comments. This research was funded by the Austrian Science Fund (FWF) [Grant DOI: 10.55776/P35519]. For open access purposes, the author has applied a CC BY public copyright license to any author accepted manuscript version arising from this submission.
}}


\author{
\large Bertram Tschiderer\thanks{
Faculty of Mathematics, University of Vienna (email: \href{mailto: bertram.tschiderer@univie.ac.at}{bertram.tschiderer@univie.ac.at}).
} 
}


\date{}


\maketitle


\begin{abstract} \small \noindent \textsc{Abstract.} An intriguing question in martingale optimal transport is to characterize the martingale with prescribed initial and terminal marginals whose transition kernel is as Gaussian as possible. In this work we address an extension of this question, in which the role of the Gaussian distribution is replaced by an arbitrary reference measure $q$. Our first main result is a dual formulation of the corresponding martingale optimization problem in terms of convex functions.

\smallskip

In the well-studied case when $q$ is Gaussian, the careful analysis of the solution to the above-mentioned optimization problem is a crucial building block in the construction of Bass martingales, i.e., Brownian martingales induced by gradients of convex functions with possibly non-degenerate starting laws. In our second main result we extend this concept beyond the Gaussian case by introducing the notion of $q$-Bass martingales in discrete time, and give sufficient conditions for their existence.

\bigskip

\small \noindent \href{https://mathscinet.ams.org/mathscinet/msc/msc2020.html}{\textit{MSC 2020 subject classifications:}} Primary 60G42, 60G44; secondary 91G20.

\bigskip

\small \noindent \textit{Keywords and phrases:} optimal transport, Brenier's theorem, Benamou--Brenier, stretched Brownian motion, Bass martingale
\end{abstract}

\section{Introduction}

\subsection{Martingale optimization problem}

Let $\mu, \nu$ be elements of $\PP_{2}(\Rd)$, the collection of probability measures on $\Rd$ with finite second moments. Assume that $\mu$ is dominated by $\nu$ in convex order, denoted by $\mu \lc \nu$, and meaning that $\int f \, d\mu \leqslant \int f\, d\nu$ holds for all convex functions $f \colon \Rd \rightarrow \R$. In the recent articles \cite{BVBHK20, CHL21, BVBST23, BVST23, AMP23}, the martingale optimization problem
\begin{equation} \label{eq_primal_qq_gamma} 
P^{\gamma}(\mu,\nu) \coloneqq \sup_{\pi \in \MT(\mu,\nu)}  \int \MCov(\pi_{x},\gamma) \, \mu(dx)
\end{equation}
received particular interest. The standard Gaussian measure $\gamma$ on $\Rd$ is used as reference measure in \eqref{eq_primal_qq_gamma} and the cost function is the maximal covariance defined as 
\begin{equation} \label{max_cov_def_qq_gamma} 
\MCov(\rho,\varrho) \coloneqq \sup_{\tilde{\pi} \in \Cpl(\rho,\varrho)} \int \langle y,z \rangle \, \tilde{\pi}(dy,dz), 
\end{equation}
for $\rho,\varrho \in \PP_{2}(\Rd)$. The maximization in \eqref{max_cov_def_qq_gamma} takes place over the set $\Cpl(\rho,\varrho)$ of all couplings $\tilde{\pi}$ between $\rho$ and $\varrho$, i.e., probability measures $\tilde{\pi}$ on $\Rd \times \Rd$ with first marginal $\rho$ and second marginal $\varrho$. In the martingale optimization problem \eqref{eq_primal_qq_gamma}, the supremum is taken over the collection $\MT(\mu,\nu)$ of martingale transports, i.e., couplings $\pi \in \Cpl(\mu,\nu)$ satisfying $\bary(\pi_{x}) \coloneqq \int y \, \pi_{x}(dy) = x$, for $\mu$-a.e.\ $x \in \Rd$. Here, the family of probability measures $\{\pi_{x}\}_{x\in\Rd}$ on $\Rd$ is obtained by disintegrating the coupling $\pi$ with respect to its first marginal $\mu$, i.e., $\pi(dx,dy) = \pi_{x}(dy) \, \mu(dx)$.

\smallskip

The optimization problem \eqref{eq_primal_qq_gamma} is a weak optimal transport problem in the sense of \cite{GRST17}, since the function $\pi \mapsto \int \MCov(\pi_{x},\gamma) \, \mu(dx)$ is non-linear, as opposed to classical optimal transport, where linear problems of the form $\pi \mapsto \int c(x,y) \, \pi(dx,dy)$ are studied. The problem \eqref{eq_primal_qq_gamma} can also be viewed as a discrete-time version of a continuous-time martingale optimization problem, see \cite{BVBHK20, BVBST23, BVST23}. Related problems also appear in \cite{Loe18, Loe23} and \cite{HT19}.

\medskip

It is natural to consider more general reference measures than the standard Gaussian $\gamma$ in \eqref{eq_primal_qq_gamma}. We thus fix a measure $q \in \PP_{2}(\Rd)$ which does not give mass to small sets, i.e., measurable sets with Hausdorff dimension at most $d-1$ in $\Rd$. This leads to the following extension of \eqref{eq_primal_qq_gamma}:
\begin{equation} \label{eq_primal_qq} 
P^{q}(\mu,\nu) 
\coloneqq  \sup_{\pi \in \MT(\mu,\nu)}  \int \MCov(\pi_{x},q) \, \mu(dx).
\end{equation}

\subsection{Dual formulation}

Our first main result is a dual formulation of the martingale optimization problem \eqref{eq_primal_qq} in terms of convex functions.

\begin{theorem} \label{theorem_new_duality_qq}
Let $\mu, \nu, q \in \PP_{2}(\Rd)$. Assume that $\mu \lc \nu$ and that $q$ does not give mass to small sets. The primal problem \eqref{eq_primal_qq} is uniquely attained. Its value $P^{q}(\mu,\nu)$ is finite and equal to
\begin{equation} \label{WeakDual_qq}
D^{q}(\mu,\nu) \coloneqq \inf_{\substack{\psi \in L^{1}(\nu), \\ \textnormal{$\psi$ convex}}} \Big( \int \psi \, d\nu - \int (\psi^{\ast} \star q)^{\ast} \, d \mu \Big).
\end{equation}
\end{theorem} 

A few comments on the notation used in the dual problem \eqref{WeakDual_qq} are in order. For a function $\psi \colon \Rd \rightarrow (-\infty,+\infty]$, its convex conjugate is denoted by $\psi^{\ast}$ and given by 
\begin{equation} \label{eq_def_conv_conj_qq}
\psi^{\ast}(x) \coloneqq \sup_{y \in \Rd} \big(\langle x,y \rangle - \psi(y)\big).
\end{equation}
For a probability measure $\varrho$ on $\Rd$, we denote by $\psi \star \varrho$ the function 
\begin{equation} \label{eq_def_star_qq}
(\psi \star \varrho)(y) \coloneqq \int \psi(y+z) \, \varrho(dz)
\end{equation}
provided $\psi(y + \, \cdot \,) \in L^{1}(\varrho)$. If $\varrho$ is symmetric (like, e.g., in the Gaussian case $\varrho = \gamma$), the function $\psi \star \varrho$ coincides with the convolution of $f$ and $\varrho$ given by
\[
(\psi \ast \varrho)(y) = \int \psi(y-z) \, \rho(dz).
\]

\smallskip

In the Gaussian case $q = \gamma$, the duality result of Theorem \ref{theorem_new_duality_qq} is contained in the statement of Theorem 1.4 in \cite{BVBST23}. It turns out that for the proof of this duality result it is not essential that the reference measure $q$ is Gaussian. In fact, it is sufficient that $q$ is a probability measure on $\Rd$ with finite second moments (so that the maximal covariance with respect to $q$ is well-defined) and which does not give mass to small sets (recall that such sets are measurable subsets of $\Rd$ with Hausdorff dimension at most $d-1$). The latter condition enables us to apply Brenier's theorem (see, e.g., \cite[Theorem 2.12]{Vil03}).

\subsection{\texorpdfstring{$q$}{q}-Bass martingales} 

We recall the following special construction of Brownian martingales (see \cite[p.\ 182]{KS98} for the classical definition) induced by gradients of convex functions with possibly non-degenerate starting laws.

\begin{definition} \label{def:BassMarti_intro.2_qq} Let $(B_{t})_{0 \leqslant t \leqslant 1}$ be $d$-dimensional Brownian motion with initial distribution $B_{0} \sim \hat{\alpha}$, where $\hat{\alpha}$ is an arbitrary probability measure on $\Rd$. Let $\hat{v} \colon \Rd \rightarrow \R$ be a convex function such that $\nabla \hat{v}(B_{1})$ is square-integrable. The martingale 
\begin{equation} \label{def:BassMarti_intro.2_condexp_qq}
\hat{M}_{t} \coloneqq 
\E[\nabla \hat{v}(B_{1}) \, \vert \, \sigma(B_{s} \colon s \leqslant t)]
= \E[\nabla \hat{v}(B_{1}) \, \vert \, B_{t}], \qquad 0 \leqslant t \leqslant 1
\end{equation}
is called \textit{Bass martingale} with \textit{Bass measure} $\hat{\alpha}$, and with initial marginal $\mu \coloneqq \Law(\hat{M}_{0})$ and terminal marginal $\nu \coloneqq \Law(\hat{M}_{1})$.
\end{definition}

Martingales of this form where introduced by Bass \cite{Bas83} (in dimension $d=1$ and with $\hat{\alpha}$ a Dirac measure) in order to derive a solution of the Skorokhod embedding problem. Definition \ref{def:BassMarti_intro.2_qq} generalizes this concept to multiple dimensions and non-degenerate starting laws; see \cite{BVBHK20, BVBST23}. 

\smallskip

We write $\gamma^{t}$ for the $d$-dimensional centered Gaussian distribution with covariance matrix $tI_{d}$ and set $\hat{v}_{t} \coloneqq \hat{v} \ast \gamma^{1-t} \colon \Rd \rightarrow \R$, for $0 \leqslant t \leqslant 1$. In these terms, \eqref{def:BassMarti_intro.2_condexp_qq} amounts to
\begin{equation} \label{def:BassMarti_intro.2_condexp_qq_ii}
\hat{M}_{t} = \nabla \hat{v}_{t}(B_{t}), \qquad 0 \leqslant t \leqslant 1.
\end{equation}
Furthermore, denoting by $\ast$ the convolution operator (either between a function and a measure or between two measures), we observe that the convex function $\hat{v}$ and the Bass measure $\hat{\alpha}$ from Definition \ref{def:BassMarti_intro.2_qq} satisfy the identities 
\begin{equation} \label{eq_def_id_bm.2_qq}
(\nabla \hat{v} \ast \gamma)(\hat{\alpha}) = \mu
\qquad \textnormal{ and } \qquad 
\nabla \hat{v}(\hat{\alpha} \ast \gamma) = \nu.
\end{equation}

\smallskip

Motivated by \eqref{eq_def_id_bm.2_qq}, it is natural to extend the definition of Bass martingales as follows.

\begin{definition} \label{def_qbm_dt} Let $\mu, \nu, q \in \PP_{2}(\Rd)$ and assume that $q$ does not give mass to small sets. A pair $(\hat{v},\hat{\alpha})$ consisting of a convex function $\hat{v} \colon \Rd \rightarrow \R$ and a probability measure $\hat{\alpha}$ on $\Rd$ is called a \textit{$q$-Bass martingale} with initial marginal $\mu$ and terminal marginal $\nu$ if
\begin{equation} \label{def_q_bm_gen}
(\nabla \hat{v} \star q)(\hat{\alpha}) = \mu
\qquad \textnormal{ and } \qquad 
\nabla \hat{v}(\hat{\alpha} \ast q) = \nu.
\end{equation}
\end{definition}

Recall that the symbol $\star$ is defined as in \eqref{eq_def_star_qq}. We summarize the fundamental relations \eqref{def_q_bm_gen} in the following graphic.
\[
\begin{tikzcd}
\hat{\alpha} \ast q  \arrow[r, "\nabla \hat{v}"] & \nu  \\
\hat{\alpha} \arrow[u, "\ast"] \arrow[r, "\nabla \hat{v} \, \star \, q"] & \mu 
\end{tikzcd}
\]
We note that $q$-Bass martingales in the sense of Definition \ref{def_qbm_dt} are discrete-time stochastic processes. In the Gaussian case $q = \gamma$, the corresponding continuous-time martingales can be represented as in \eqref{def:BassMarti_intro.2_condexp_qq}, \eqref{def:BassMarti_intro.2_condexp_qq_ii}. We leave the definition of general $q$-Bass martingales in continuous time as an open question.

\smallskip

Next, we define an appropriate class of convex functions $\hat{v} \colon \Rd \rightarrow \R$ which generate $q$-Bass martingales in a natural way.

\begin{definition} \label{def_psi_v_qbg} Let $\mu, q \in \PP_{2}(\Rd)$ and assume that $q$ does not give mass to small sets. Let $\hat{v} \colon \Rd \rightarrow \R$ be a convex function and denote by $\hat{\psi} \coloneqq \hat{v}^{\ast}$ its convex conjugate. We say that $\hat{v}$ is \textit{$q$-Bass martingale generating} with initial marginal $\mu$ if
\begin{enumerate}[label=(\roman*)] 
\item \label{def_psi_v_qbg_i} $\mu(\operatorname{int}(\dom\hat{\psi})) = 1$,
\item \label{def_psi_v_qbg_ii} the function $\hat{v} \star q$ is finite-valued, strictly convex, and $\nabla(\hat{v} \star q) = (\nabla \hat{v}) \star q$,
\item \label{def_psi_v_qbg_iii} the measure $\nu^{\hat{v}} \coloneqq \nabla \hat{v}(\hat{\alpha}^{\hat{v}} \ast q)$ has finite second moment, where $\hat{\alpha}^{\hat{v}} \coloneqq \nabla(\hat{v} \star q)^{\ast}(\mu)$.
\end{enumerate}
\end{definition}

In our second main result we show that $q$-Bass martingale generating functions indeed induce $q$-Bass martingales in the natural way. Beyond that, we establish the connection to the primal problem \eqref{eq_primal_qq} as well as the dual problem \eqref{WeakDual_qq}.

\begin{theorem} \label{qq_smr_qq} Let $\mu, q \in \PP_{2}(\Rd)$ and assume that $q$ does not give mass to small sets. Let $\hat{v}$ be $q$-Bass martingale generating with initial marginal $\mu$ and let $\nu^{\hat{v}}, \hat{\alpha}^{\hat{v}}$ be as in Definition \ref{def_psi_v_qbg}.
\begin{enumerate}[label=(\arabic*)] 
\item \label{qq_smr_qq_i} The pair $(\hat{v},\hat{\alpha}^{\hat{v}})$ is a $q$-Bass martingale with initial marginal $\mu$ and with terminal marginal $\nu^{\hat{v}}$.
\item \label{qq_smr_qq_ii} The optimizer $\hat{\pi}(dx,dy) = \hat{\pi}_{x}(dy) \, \mu(dx)$ of the primal problem \eqref{eq_primal_qq} between $\mu$ and $\nu = \nu^{\hat{v}}$ is given by
\[
\hat{\pi}_{x} = \nabla \hat{v}\Big(\nabla(\hat{v} \star q)^{\ast}(x) + \, \cdot \, \Big)(q),
\]
for $\mu$-a.e.\ $x \in \Rd$.
\item \label{qq_smr_qq_iii} The function $\hat{\psi} = \hat{v}^{\ast}$ is an optimizer of the dual problem \eqref{WeakDual_qq} between $\mu$ and $\nu = \nu^{\hat{v}}$.
\end{enumerate}
\end{theorem}

The proof of Theorem \ref{qq_smr_qq} is presented in Subsection \ref{qq_smr_qq_sec_15} of Section \ref{chap_qbms}. 

\begin{remark} \
\begin{enumerate}[label=(\roman*)] 
\item The conditions \ref{def_psi_v_qbg_i} -- \ref{def_psi_v_qbg_iii} of Definition \ref{def_psi_v_qbg} do not necessarily imply that the function $\hat{\psi} = \hat{v}^{\ast}$ is integrable with respect to the probability measure $\nu^{\hat{v}}$. Therefore attainment of the dual problem \eqref{WeakDual_qq} as claimed in part \ref{qq_smr_qq_iii} of Theorem \ref{qq_smr_qq} has to be understood in a ``relaxed'' sense frequently encountered in martingale transport problems \cite{BJ16, BNT17, BNS22}; see Definition \ref{def_dual_opt_qq} below.
\item Already in the Gaussian case $q = \gamma$, the conditions \ref{def_psi_v_qbg_i} -- \ref{def_psi_v_qbg_iii} of Definition \ref{def_psi_v_qbg} are not only sufficient but also necessary for a result in the form of Theorem \ref{qq_smr_qq} to be true. More precisely, if $q = \gamma$, it follows from the results of \cite{BVBST23} that a function $\hat{v}$ satisfying the properties \ref{qq_smr_qq_i} --  \ref{qq_smr_qq_iii} of Theorem \ref{qq_smr_qq} has to be $\gamma$-Bass martingale generating in the sense of Definition \ref{def_psi_v_qbg}. \hfill $\diamond$
\end{enumerate}
\end{remark}

Finally, we raise the question under which assumptions a $q$-Bass martingale with given marginals $\mu$ and $\nu$ exists. For $\gamma$-Bass martingales this question was answered in \cite[Theorem 1.3]{BVBST23}: a $\gamma$-Bass martingale with marginals $\mu$ and $\nu$ exists if and only if the pair $(\mu,\nu)$ is irreducible in the following sense.

\begin{definition} For probability measures $\mu, \nu$ we say that the pair $(\mu,\nu)$ is \textit{irreducible} if for all measurable sets $A, B \subseteq \Rd$ with $\mu(A)> 0$ and $\nu(B)>0$ there is a martingale $(X_{t})_{0 \leqslant t \leqslant 1}$ with $X_{0} \sim \mu$, $X_{1} \sim \nu$ such that $\mathds{P}(X_{0}\in A, X_{1}\in B) >0$. 
\end{definition} 

Given an irreducible pair $(\mu,\nu)$, we conjecture that a $q$-Bass martingale with marginals $\mu$ and $\nu$ exists, provided that $q$ is absolutely continuous and satisfies suitable integrability conditions.

\subsection{Related literature} 

The formulation of optimal transport as a field in mathematics goes back to Monge \cite{Mon81} and Kantorovich \cite{Kan42}. The modern theory is based on the seminal contributions of Benamou, Brenier, and McCann \cite{Bre87, Bre91, McC94, McC95, BB99}. We refer to the books \cite{Vil03, Vil09, AG13, San15} for introductions to the field of optimal transport as well as a variety of applications in different areas.

\smallskip 

In martingale optimal transport (see e.g.\ \cite{HN12, BHLP13, TT13, DS14, GHLT14, CLM17} among many others), the transport plan satisfies an additional martingale constraint. This additional requirement leads to notable consequences in mathematical finance \cite{BHLP13}, the study of martingale inequalities \cite{BN15, OST15, HLOST16}, and the Skorokhod embedding problem \cite{BCH17, KTT17, BNS22}. For structural descriptions of optimal martingale transport plans we refer to \cite{BJ16, OS17, DMT19, GKL19}. 

\smallskip

Interesting problems in the field of martingale optimal transport are formulated in continuous time, see e.g.\ \cite{DS14, BHLT17, COT19, GKP19, GLW19, GKL20, CKPS21, GL21}. Most notably, the martingale version of the Benamou--Brenier problem was introduced in \cite{BVBHK20} in probabilistic language and in \cite{HT19} in PDE language. It was then analyzed from the point of view of duality theory in \cite{BVBST23} and complemented by a variational perspective in \cite{BVST23}. In the context of market impact in finance, the same kind of problem appeared independently in \cite{Loe18, Loe23}. For the interpretation of the martingale Benamou--Brenier problem as a local volatility model we refer to the recent articles \cite{BVBHK20, CHL21, AMP23}.

\section{Proof of the duality result} \label{p_o_t_d_r_qq}

Throughout the rest of this work we fix $\mu, \nu \in \PP_{2}(\Rd)$ with $\mu \lc \nu$. We also fix a reference measure $q \in \PP_{2}(\Rd)$ which does not give mass to small sets. The goal of this chapter is to provide the proof of the duality result as stated in Theorem \ref{theorem_new_duality_qq}. The proof is based on \cite[Theorem 1.4]{BVBST23}, which is formulated in the Gaussian setting $q = \gamma$.

\subsection{A first dual formulation} \label{p_o_t_d_r_qq_a}

To formulate a first version of a dual problem to \eqref{eq_primal_qq}, we consider the set of continuous functions with quadratic growth 
\[
\Cqs \coloneqq \big\{ \psi \colon \Rd \rightarrow \R \textnormal{ continuous s.t.\ } \exists \, a,k,\ell \in \R \textnormal{ with } 
\ell + \tfrac{\vert \, \cdot \,  \vert^{2}}{2} \leqslant \psi(\, \cdot \,) \leqslant a + k \vert \cdot  \vert^{2} \big\}.
\]
We define the dual problem 
\begin{equation} \label{eq_dual_qq} 
\tilde{D}^{q}(\mu,\nu) \coloneqq 
\inf_{\psi \in \Cqs} \Big( \int \psi \, d\nu - \int \varphi^{\psi} \, d \mu \Big),
\end{equation}
where the function $\Rd \ni x \longmapsto \varphi^{\psi}(x)$ is given by
\begin{equation} \label{eq_phi_psi_x_qq}
\varphi^{\psi}(x) \coloneqq \inf_{p \in \PP_{2}^{x}(\Rd)} \Big( \int \psi \, dp - \MCov(p,q)\Big).
\end{equation}
Here, the set $\PP_{2}^{x}(\Rd)$ denotes the elements $p$ of $\PP_{2}(\Rd)$ with barycenter
\[
\bary(p) = \int y \, p(dy) = x \in \Rd.
\]

\begin{proposition} \label{prop_no_duality_gap_first_qq} There is no duality gap between the primal problem \eqref{eq_primal_qq} and the dual problem \eqref{eq_dual_qq}, i.e., $P^{q}(\mu,\nu) = \tilde{D}^{q}(\mu,\nu)$. Moreover, the primal problem is uniquely attained and has a finite value, i.e., there exists a unique $\hat{\pi} \in \MT(\mu,\nu)$ such that
\begin{equation} \label{eq_the_no_duality_gap_qq}
P^{q}(\mu,\nu) 
= \int \MCov(\hat{\pi}_{x},q) \, \mu(dx) 
< +\infty.
\end{equation}
\begin{proof} For $x \in \Rd$ and $p \in \PP_{2}(\Rd)$ we define the cost function
\[
C(x,p) \coloneqq
\begin{cases}
-\MCov(p,q) + \tfrac{1}{2} \int \vert y \vert^{2} \, dp, & \textnormal{ if } \bary(p) = x,\\
+\infty, & \textnormal{ if } \bary(p) \neq x,
\end{cases}
\]
which is bounded from below and convex in the second argument. If we equip $\PP_{2}(\Rd)$ with the topology induced by the quadratic Wasserstein distance, one verifies that the function $(x,p) \mapsto C(x,p)$ is jointly lower semicontinuous with respect to the product topology on $\Rd \times \PP_{2}(\Rd)$. Hence we can apply \cite[Theorem 1.2]{BVBP19}, which guarantees the existence of an optimizer $\hat{\pi} \in \MT(\mu,\nu)$ of the primal problem \eqref{eq_primal_qq}. This optimizer is unique, as the value $P^{q}(\mu,\nu)$ is finite and $(x,p) \mapsto C(x,p)$ is strictly convex in the second argument. The former follows from the fact that $\mu$ and $\nu$ have finite second moments, the latter is a consequence of Brenier's theorem.

\smallskip

In order to show that there is no duality gap, we introduce the space of continuous functions which are bounded from below and have at most quadratic growth
\[
\Cbqs \coloneqq \big\{ \tilde{\psi} \colon \Rd \rightarrow \R \textnormal{ continuous s.t.\ } \exists \, a,k,\ell \in \R \textnormal{ with } \ell \leqslant \tilde{\psi}(\, \cdot \,) \leqslant a + k \vert \cdot  \vert^{2} \big\}.
\]
Then by \cite[Theorem 1.3]{BVBP19} the value $P^{q}(\mu,\nu)$ of the primal problem is equal to
\[
\tilde{D}_{\textnormal{b},2}^{q}(\mu,\nu) \coloneqq \inf_{\tilde{\psi} \in \Cbqs} \Big( \int \big( \tilde{\psi}(\, \cdot \,) + \tfrac{\vert \, \cdot \, \vert^{2}}{2}\big) \, d\nu - \int \tilde{\varphi}^{\tilde{\psi}} \, d\mu \Big),
\]
where
\[
\tilde{\varphi}^{\tilde{\psi}}(x) \coloneqq \inf_{p \in \PP_{2}^{x}(\Rd)} \Big( \int \big( \tilde{\psi}(\, \cdot \,) + \tfrac{\vert \, \cdot \, \vert^{2}}{2}\big) \, dp - \MCov(p,q)\Big).
\]
Finally, passing from the functions $\tilde{\psi} \in \Cbqs$ to $\psi(\, \cdot \,) \coloneqq \tilde{\psi}(\, \cdot \,) + \frac{\vert \, \cdot \, \vert^{2}}{2} \in \Cqs$, we see that $\tilde{D}_{\textnormal{b},2}^{q}(\mu,\nu) = \tilde{D}^{q}(\mu,\nu)$.
\end{proof}
\end{proposition}

We close this section by providing a heuristic argument why \eqref{eq_dual_qq} is indeed a plausible dual formulation of \eqref{eq_primal_qq}. Considering the ``indicator function''
\[
\chi(\pi) \coloneqq 
\inf_{\psi \in \Cqs}  \Big( \int \psi \, d \nu - 
\int \psi \, d\pi_{x} \, d\mu(x)  \Big)
= 
\begin{cases} 
0, & \mbox{if }  \int \pi_{x} \, d\mu(x) = \nu, \\
-\infty, & \mbox{else},
\end{cases} 
\] 
we formally obtain the desired duality relation by interchanging $\inf$ and $\sup$, to wit
\begin{multline*}
P^{q}(\mu,\nu)
=
\sup_{\substack{\pi(dx,dy) = \pi_{x}(dy) \, \mu(dx),  \\ \pi_{x} \in \PP_{2}^{x}(\Rd)}} \, 
\Big(\int \MCov(\pi_{x},q) \, d \mu(x) + \chi(\pi)\Big) \\
= \inf_{\psi \in \Cqs} 
\bigg( \int \sup_{\pi_{x} \in \PP_{2}^{x}(\Rd)} 
\Big( \MCov(\pi_{x},q) + \int \psi \, d(\nu -\pi_{x}) \Big) 
\, d \mu(x) \bigg) 
= \tilde{D}^{q}(\mu,\nu). 
\end{multline*}

\subsection{Convexity of the dual variables} \label{p_o_t_d_r_qq_b}

In this section we make the crucial observation that in the dual problem \eqref{eq_dual_qq} it suffices to optimize over the class of functions $\psi \in \Cqs$ which are convex. We then show that it is also equivalent to optimize over all convex functions $\psi \colon \Rd \rightarrow (-\infty,+\infty]$ which are only $\mu$-a.s.\ finite, but not necessarily of quadratic growth.

\smallskip

Given some $\psi \in \Cqs$, we denote by $\conv\psi$ the convex hull of $\psi$, i.e., the greatest convex function smaller or equal to $\psi$. It will be convenient to have an explicit representation for the convex hull of a function, as in \eqref{lem_conv_eq_1_qq} below. This identity is usually stated in the more specific form 
\[
(\conv\psi)(y) 
= \inf \bigg\{ \sum_{i=1}^{d+1} \lambda_{i} \psi(y_{i}) \colon \sum_{i=1}^{d+1} \lambda_{i} y_{i} = y \bigg\},
\]
where the infimum is taken over all expressions of $y$ as a convex combination of $d+1$ points, see \cite[Corollary 17.1.5]{Roc70}. 

\begin{lemma} \label{lem_conv_qq} For a function $\psi \in \Cqs$ its convex hull $\conv\psi$ is given by
\begin{equation} \label{lem_conv_eq_1_qq}
\Rd \ni y \longmapsto (\conv\psi)(y) = \inf_{p \in \PP_{2}^{y}(\Rd)} \int \psi \, dp
\end{equation}
and again $\conv\psi \in \Cqs$.
\end{lemma}

Recalling the dual problem \eqref{eq_dual_qq}, we define the dual function 
\begin{equation} \label{eq_def_dual_func_qq}
\Cqs \ni \psi \longmapsto \tilde{\mathcal{D}}(\psi) \coloneqq \int \psi \, d\nu - \int \varphi^{\psi} \, d \mu.
\end{equation}
Now we prove that it suffices to optimize the dual function over the class of functions $\psi \in \Cqs$ which are convex. 

\begin{proposition} \label{prop_crucial_qq} For all $\psi \in \Cqs$ we have $\tilde{\mathcal{D}}(\conv \psi) \leqslant \tilde{\mathcal{D}}(\psi)$ and consequently
\begin{equation} \label{eq_prop_crucial_m_qq}
\tilde{D}^{q}(\mu,\nu) 
= \inf_{\psi \in \Cqs} \tilde{\mathcal{D}}(\psi) 
= \inf_{\substack{\psi \in \Cqs, \\ \textnormal{$\psi$ convex}}}\tilde{\mathcal{D}}(\psi).
\end{equation}
\begin{proof} Let $\varepsilon > 0$, $\psi \in \Cqs$ and $\{p_{x}\}_{x \in \Rd} \subseteq \PP_{2}(\Rd)$ be a measurable collection of probability measures with $\bary(p_{x}) = x$. To show the claim, it is sufficient to construct a measurable family $\{\bar{p}_{x}\}_{x \in \Rd} \subseteq \PP_{2}(\Rd)$ with $\bary(\bar{p}_{x}) = x$ such that
\begin{equation} \label{prop_crucial_01.qq}
\MCov(p_{x},q) + \int \conv \psi \, d(\nu-p_{x}) 
\leqslant \MCov(\bar{p}_{x},q) + \int \psi \, d(\nu-\bar{p}_{x}) + \varepsilon.
\end{equation}
By Lemma \ref{lem_conv_qq} and a measurable selection argument we can choose a measurable collection of probability measures $\{\tilde{p}_{y}\}_{y \in \Rd} \subseteq \PP_{2}(\Rd)$ with $\bary(\tilde{p}_{y}) = y$ such that 
\begin{equation} \label{prop_crucial_02.qq}
\int \psi \, d\tilde{p}_{y}  \leqslant  (\conv\psi)(y) + \varepsilon.
\end{equation}
Then we define $\bar{p}_{x}(dz) \coloneqq \int_{y} \tilde{p}_{y}(dz) \, p_{x}(dy)$, so that $\bary(\bar{p}_{x}) = x$. Integrating \eqref{prop_crucial_02.qq} with respect to $p_{x}(dy)$ yields
\begin{equation} \label{prop_crucial_03.qq}
\int \psi \, d\bar{p}_{x}  
\leqslant  \int \conv\psi \, dp_{x} + \varepsilon.
\end{equation}
Since $\psi, \conv \psi \in \Cqs$ and $p_{x} \in \PP_{2}^{x}(\Rd)$ we conclude
\[
\ell + \tfrac{1}{2} \int \vert y  \vert^{2} \, \bar{p}_{x}(dy)
\leqslant a + k \int \vert y  \vert^{2} \, p_{x}(dy) +  \varepsilon < + \infty,
\]
so that $\bar{p}_{x} \in \PP_{2}^{x}(\Rd)$.

\smallskip

In order to show the inequality \eqref{prop_crucial_01.qq}, we first observe that $p_{x} \lc \bar{p}_{x}$ by Jensen's inequality. Together with the Kantorovich duality (see, e.g., \cite[Theorem 5.10]{Vil09}), we conclude that\footnote{In fact, $p_{x} \lc \bar{p}_{x}$ is equivalent to the inequality $\MCov(p_{x},q) \leqslant \MCov(\bar{p}_{x},q)$ being valid for all probability measures $q \in \PP_{2}(\Rd)$; see., e.g., \cite[Theorem 1]{AP22} or \cite[Corollary 1.2]{WZ23}.}
\begin{equation} \label{prop_crucial_04.qq}
\begin{aligned}
\MCov(p_{x},q) 
&= \inf_{\substack{f \colon \Rd \rightarrow \R \\ \textnormal{convex}}} 
\Big( \int f \, d p_{x} + \int f^{\ast} \, dq \Big) \\
&\leqslant \inf_{\substack{f \colon \Rd \rightarrow \R \\ \textnormal{convex}}} 
\Big( \int f \, d \bar{p}_{x} + \int f^{\ast} \, dq \Big)
= \MCov(\bar{p}_{x},q).
\end{aligned}
\end{equation}
On the other hand, from $\conv \psi \leqslant \psi$ and \eqref{prop_crucial_03.qq} we have the inequality 
\begin{equation} \label{prop_crucial_05.qq}
\int \conv \psi \, d(\nu-p_{x})  
\leqslant \int \psi \, d(\nu-\bar{p}_{x}) + \varepsilon.
\end{equation}
Finally, summing \eqref{prop_crucial_04.qq} and \eqref{prop_crucial_05.qq}, we obtain the inequality \eqref{prop_crucial_01.qq}.
\end{proof}
\end{proposition}

Next, we extend the definition \eqref{eq_def_dual_func_qq} of the dual function $\tilde{\mathcal{D}}(\, \cdot \,)$ to the class of convex functions $\psi \colon \Rd \rightarrow (-\infty,+\infty]$ which are not confined to be in $\Cqs$, but which are only required to satisfy $\mu(\dom \psi) = 1$, where 
\[
\dom \psi \coloneqq \{ y \in \Rd \colon \psi(y) < + \infty \}
\]
denotes the domain of $\psi$. In order for the difference of the integrals in \eqref{eq_def_dual_func_qq} still to be well-defined, the dual function has to be understood in a ``relaxed'' sense, as frequently encountered in martingale transport problems; see \cite{BJ16, BNT17, BNS22}. More precisely, we define the ``relaxed'' representation of \eqref{eq_def_dual_func_qq} by
\begin{equation} \label{eq_rep_dual_func_qq}
\tilde{\mathcal{E}}(\psi) \coloneqq \int  \Big( \int \psi(y) \, \pi_{x}(dy)  - \varphi^{\psi}(x) \Big) \, \mu(dx),
\end{equation}
where $\pi(dx,dy) = \pi_{x}(dy) \, \mu(dx)$ is an arbitrary fixed element of $\MT(\mu,\nu)$. Then $\tilde{\mathcal{E}}(\psi)$ is well-defined for every convex function $\psi$ satisfying $\mu(\dom \psi) = 1$. Taking the infimum of \eqref{eq_rep_dual_func_qq} over all such $\psi$ again leads to the same value as in \eqref{eq_prop_crucial_m_qq}. This is summarized in the following lemma.

\begin{lemma} \label{lem_ext_def_dual.qq} Fix $\pi \in \MT(\mu,\nu)$ and let $\psi \colon \Rd \rightarrow (-\infty,+\infty]$ be a convex function with $\mu(\dom \psi) = 1$. Then $\tilde{\mathcal{E}}(\psi) \in [0,+\infty]$ and we have the inequality
\begin{equation} \label{lem_ext_dual_func_01.qq}
\tilde{D}^{q}(\mu,\nu) \leqslant \tilde{\mathcal{E}}(\psi).
\end{equation}
In particular, recalling \eqref{eq_prop_crucial_m_qq}, we have
\begin{equation} \label{lem_ext_dual_func_02.qq}
\tilde{D}^{q}(\mu,\nu) 
= \inf_{\substack{\psi \in \Cqs, \\ \textnormal{$\psi$ convex}}}\tilde{\mathcal{D}}(\psi) 
= \inf_{\substack{\mu(\dom \psi) = 1, \\ \textnormal{$\psi$ convex}}}\tilde{\mathcal{E}}(\psi).
\end{equation}
\begin{proof} Fix $\pi \in \MT(\mu,\nu)$ and let $\psi \colon \Rd \rightarrow (-\infty,+\infty]$ be a convex function with $\mu(\dom \psi) = 1$. First, note that by Jensen's inequality we have
\[
\int \psi(y) \, \pi_{x}(dy) 
\geqslant \psi\Big(\int y \, \pi_{x}(dy)\Big) 
= \psi(x), 
\]
and by taking $p = \delta_{x}$ in \eqref{eq_phi_psi_x_qq} we obtain $\varphi^{\psi}(x) \leqslant \psi(x)$. Hence
\begin{equation} \label{eq_ext_dual_func_03.qq}
\int \psi(y) \, \pi_{x}(dy)  - \varphi^{\psi}(x) \geqslant 0,
\end{equation}
for $\mu$-a.e.\ $x \in \Rd$, so that $\tilde{\mathcal{E}}(\psi) \in [0,+\infty]$. In order to prove the inequality \eqref{lem_ext_dual_func_01.qq}, we distinguish two cases. In the case $\int \psi \, d \pi_{x} = + \infty$, for $x \in \Rd$ in a set of positive $\mu$-measure, we conclude from \eqref{eq_ext_dual_func_03.qq} that $\tilde{\mathcal{E}}(\psi) = +\infty$, so that \eqref{lem_ext_dual_func_01.qq} is trivially satisfied. Now suppose that $\int \psi \, d \pi_{x} < + \infty$, for $\mu$-a.e.\ $x \in \Rd$. Recall from Proposition \ref{prop_no_duality_gap_first_qq} that there exists an optimizer $\hat{\pi} \in \MT(\mu,\nu)$ of the primal problem \eqref{eq_primal_qq}. Taking $p = \hat{\pi}_{x}$ in \eqref{eq_phi_psi_x_qq} we get
\[
\varphi^{\psi}(x) \leqslant  \int \psi \, d\hat{\pi}_{x} - \MCov(\hat{\pi}_{x},q)
\]
and therefore 
\[
\tilde{\mathcal{E}}(\psi) \geqslant 
\int \Big(  \int \psi \, d \pi_{x} + \MCov(\hat{\pi}_{x},q) - \int \psi \, d\hat{\pi}_{x}  \Big) \, \mu(dx) 
= P^{q}(\mu,\nu),
\]
where we also used that both $\hat{\pi}$ and $\pi$ have $\nu$ as second marginal, and that $\hat{\pi}$ is a primal optimizer, i.e., satisfying \eqref{eq_the_no_duality_gap_qq}. Since $P^{q}(\mu,\nu) = \tilde{D}^{q}(\mu,\nu)$ by Proposition \ref{prop_no_duality_gap_first_qq}, we again see that \eqref{lem_ext_dual_func_01.qq} does hold.

\smallskip

Finally, from \eqref{lem_ext_dual_func_01.qq} and recalling \eqref{eq_prop_crucial_m_qq}, we conclude
\[
\tilde{D}^{q}(\mu,\nu) 
\leqslant \inf_{\substack{\mu(\dom \psi) = 1, \\ \textnormal{$\psi$ convex}}}\tilde{\mathcal{E}}(\psi) 
\leqslant \inf_{\substack{\psi \in \Cqs, \\ \textnormal{$\psi$ convex}}}\tilde{\mathcal{D}}(\psi) 
= \tilde{D}^{q}(\mu,\nu),
\]
which shows \eqref{lem_ext_dual_func_02.qq}. 
\end{proof}
\end{lemma}

\subsection{Proof of Theorem \texorpdfstring{\ref{theorem_new_duality_qq}}{1.1}} \label{p_o_t_d_r_qq_c}

According to Proposition \ref{prop_no_duality_gap_first_qq}, there is no duality gap between the primal problem \eqref{eq_primal_qq} and the dual problem \eqref{eq_dual_qq}, i.e., $P^{q}(\mu,\nu) = \tilde{D}^{q}(\mu,\nu)$. Hence for the proof of Theorem \ref{theorem_new_duality_qq} we have to show that
\begin{equation} \label{p_o_t_d_r_qq_c.eq}
\tilde{D}^{q}(\mu,\nu) = \inf_{\substack{\psi \in L^{1}(\nu), \\ \textnormal{$\psi$ convex}}} \Big( \int \psi \, d\nu - \int (\psi^{\ast} \star q)^{\ast} \, d \mu \Big).
\end{equation}
Recalling the definitions \eqref{eq_dual_qq}, \eqref{eq_phi_psi_x_qq}, and \eqref{eq_def_dual_func_qq}, it should come as no surprise that the main idea behind the proof of the identity \eqref{p_o_t_d_r_qq_c.eq} is to apply Proposition \ref{prop_crucial_qq} and to show that $\varphi^{\psi} = (\psi^{\ast} \star q)^{\ast}$ holds for every convex function $\psi \in \Cqs$. This motivates our next goal, namely to solve the minimization problem \eqref{eq_phi_psi_x_qq}, which we rewrite as a maximization problem, to wit
\begin{equation} \label{eq_phi_psi_x_max.qq}
-\varphi^{\psi}(x) = \sup_{p \in \PP_{2}^{x}(\Rd)} \Big( \MCov(p,q) - \int \psi \, dp\Big).
\end{equation}
As a preliminary step, we consider the simpler problem
\begin{equation} \label{eq_phi_psi.qq}
\phi^{\psi} \coloneqq \sup_{p \in \PP_{2}(\Rd)} \Big( \MCov(p,q) - \int \psi \, dp \Big),
\end{equation}
where we do not prescribe the barycenter $x \in \Rd$ of $p \in \PP_{2}(\Rd)$. In Lemma \ref{lem_phi_psi_gen.qq} below we will show for an arbitrary proper convex function $\psi \colon \Rd \rightarrow (-\infty,+\infty]$, that the value $\phi^{\psi}$ of \eqref{eq_phi_psi.qq} is equal to $\int \psi^{\ast} \, dq$. Recall that $\psi$ is called a proper convex function if $\psi$ is convex and $\dom \psi \neq \varnothing$. 

\smallskip

Solving the maximization problem \eqref{eq_phi_psi.qq} leads to an interesting connection with Brenier maps, see Lemma \ref{lem_phi_psi.qq} below. By Brenier's theorem (see, e.g., \cite[Theorem 2.12]{Vil03}), the optimal transport for quadratic cost between $q \in \PP_2(\Rd)$ and $p \in \PP_2(\Rd)$ is induced by the $q$-a.e.\ defined gradient $\nabla v$ (the ``Brenier map'') of some convex function $v \colon \Rd \rightarrow \R$ via $(\nabla v)(q) = p$. 

\begin{lemma} \label{lem_phi_psi.qq} Let $v \colon \Rd \rightarrow \R$ be a finite-valued convex function and $\psi \coloneqq v^{\ast}$ its convex conjugate. Assume that the probability measure $\hat{p} \coloneqq (\nabla v)(q)$ has finite second moment. Then $\hat{p}$ is the unique maximizer of the optimization problem \eqref{eq_phi_psi.qq} and
\[
\phi^{\psi} = \int v \, dq = \int \psi^{\ast} \, dq < + \infty.
\]
\begin{proof} For $p \in \PP_2(\Rd)$, we denote by $T_{q}^{p}$ the Brenier map from $q$ to $p \in \PP_{2}(\Rd)$ and note that $T_{q}^{\hat{p}} = \nabla v$ by Brenier's theorem. Furthermore, we have that
\[
v(z) = 
\sup_{y \in \Rd} \big(\langle y,z \rangle - v^{\ast}(y)\big) 
= \big\langle \nabla v(z),z \big\rangle - v^{\ast}\big( \nabla v(z)\big),
\]
for $q$-a.e.\ $z \in \Rd$. Using these observations, for $p \in \PP_{2}(\Rd)$ we get
\begin{align*}
\MCov(p,q) - \int \psi \, dp 
&= \int \Big( \big\langle T_{q}^{p}(z), z\big\rangle 
-  v^{\ast}\big(T_{q}^{p}(z)\big) \Big)  \, q(dz) \\
&\leqslant \int \sup_{y \in \Rd} \big( \langle y, z\rangle 
-  v^{\ast}(y) \big)  \, q(dz) 
= \int v(z)  \, q(dz) \\
&= \int \Big( \big\langle \nabla v(z),z \big\rangle - v^{\ast}\big( \nabla v(z)\big) \Big)  \, q(dz) \\
&= \MCov(\hat{p},q) - \int \psi \, d\hat{p},
\end{align*}
with equality if and only if $T_{q}^{p}(z) = T_{q}^{\hat{p}}(z)$, for $q$-a.e.\ $z \in \Rd$. This in turn is the case if and only if $p = \hat{p}$. Finally, from the convexity of $v$ and the Cauchy--Schwarz inequality we obtain
\[
\int \vert v \vert \, d q \leqslant \vert v(0) \vert + 
\sqrt{\int \vert \nabla v \vert^{2} \, d q} \
\sqrt{\int \vert z \vert^{2} \, d q(z)}
< + \infty,
\]
which proves that $\phi^{\psi} < +\infty$. 
\end{proof}
\end{lemma}

Next, we relax the rather strong assumptions of Lemma \ref{lem_phi_psi.qq}. If we are given just a proper convex function $\psi \colon \R^{d} \rightarrow (-\infty,+\infty]$, we can still compute the value of the supremum in \eqref{eq_phi_psi.qq}, without explicitly constructing a maximizer of this optimization problem.

\begin{lemma} \label{lem_phi_psi_gen.qq} Let $\psi \colon \Rd \rightarrow (-\infty,+\infty]$ be a proper convex function. Then
\begin{equation} \label{lem_phi_psi_gen_04.qq}
\phi^{\psi} = \int \psi^{\ast}  \, dq.
\end{equation}
\begin{proof} Using probabilistic notation, we rewrite the supremum in \eqref{eq_phi_psi.qq} as
\begin{equation} \label{lem_phi_psi_gen_01.qq}
\phi^{\psi} = \sup \E\big[ \langle Y, Z \rangle - \psi(Y)\big],
\end{equation}
where the supremum in \eqref{lem_phi_psi_gen_01.qq} is taken over all probability spaces such that $Z \sim q$ and $Y$ is an $\Rd$-valued random variable with finite second moment. Replacing $Y$ by $\E[Y \, \vert \, Z]$, we observe that the maximization in \eqref{lem_phi_psi_gen_01.qq} can be restricted to random variables $Y$ which are measurable functions of $Z$, and we obtain
\begin{equation} \label{lem_phi_psi_gen_02.qq}
\phi^{\psi} 
= \sup_{Y \in L^{2}(q;\Rd)} 
\int \Big( \big\langle Y(z), z \big\rangle - \psi\big(Y(z)\big) \Big) \, q(dz),
\end{equation}
where $L^{2}(q;\Rd)$ denotes the space of $\Rd$-valued Borel measurable functions on $\Rd$, which are square-integrable under $q$. Clearly, for any $Y \in L^{2}(q;\Rd)$ we have 
\[
\int \Big( \big\langle Y(z), z \big\rangle - \psi\big(Y(z)\big) \Big) \, q(dz)
\leqslant \int \sup_{y \in \Rd} \big(\langle y,z \rangle - \psi(y)\big) \, q(dz) 
= \int \psi^{\ast}(z) \, q(dz),
\]
which shows the inequality $\phi^{\psi} \leqslant \int \psi^{\ast}  \, dq$. In order to see the reverse inequality, we define the auxiliary problem
\begin{equation} \label{lem_phi_psi_gen_03.qq}
\phi_{\infty}^{\psi} 
\coloneqq \sup_{Y \in L^{\infty}(q;\Rd)} 
\int \Big( \big\langle Y(z), z \big\rangle - \psi\big(Y(z)\big) \Big) \, q(dz),
\end{equation}
where $L^{\infty}(q;\Rd)$ denotes the space of $\Rd$-valued Borel measurable functions on $\Rd$, which are bounded $q$-a.e. Comparing \eqref{lem_phi_psi_gen_02.qq} with \eqref{lem_phi_psi_gen_03.qq}, we obviously have $\phi^{\psi} \geqslant \phi_{\infty}^{\psi}$. Now we claim that
\begin{equation} \label{lem_phi_psi_gen_05.qq}
\phi_{\infty}^{\psi} 
\geqslant \int \sup_{y \in \Rd} \big(\langle y,z \rangle - \psi(y)\big) \, q(dz) 
= \int \psi^{\ast}(z) \, q(dz),
\end{equation}
which will finish the proof of \eqref{lem_phi_psi_gen_04.qq}. To see this, we first write
\[
\phi_{\infty}^{\psi} 
= \lim_{N \rightarrow \infty} \, \sup_{\substack{Y \in L^{\infty}(q;\Rd), \\ \vert Y  \vert \leqslant N}} \, 
\int \Big( \big\langle Y(z), z \big\rangle - \psi\big(Y(z)\big) \Big) \, q(dz).
\]
Using a measurable selection argument, we obtain
\[
\sup_{\substack{Y \in L^{\infty}(q;\Rd), \\ \vert Y  \vert \leqslant N}} \, 
\int \Big( \big\langle Y(z), z \big\rangle - \psi\big(Y(z)\big) \Big) \, q(dz) 
\geqslant
\int \sup_{\substack{y \in \Rd, \\ \vert y  \vert \leqslant N}} \,  \big(\langle y,z \rangle - \psi(y)\big) \, q(dz).
\]
Since $\psi$ is proper we can choose $y_{0} \in \dom \psi \neq \varnothing$. Then for $N$ large enough we have
\[
\sup_{\substack{y \in \Rd, \\ \vert y  \vert \leqslant N}} \,  \big(\langle y,z \rangle - \psi(y)\big) 
\geqslant \langle y_{0},z \rangle - \psi(y_{0}),
\]
with the right-hand side being integrable with respect to $q(dz)$. Hence we can apply the monotone convergence theorem and deduce that
\[
\lim_{N \rightarrow \infty} \int \sup_{\substack{y \in \Rd, \\ \vert y  \vert \leqslant N}} \,  \big(\langle y,z \rangle - \psi(y)\big) \, q(dz)
= \int \sup_{y \in \Rd} \big(\langle y,z \rangle - \psi(y)\big) \, q(dz),
\]
which completes the proof of \eqref{lem_phi_psi_gen_05.qq}.
\end{proof}
\end{lemma}

Lemma \ref{lem_phi_psi_gen.qq} is an important auxiliary result for the proof of Theorem \ref{theorem_new_duality_qq}, more precisely, for the proof of the identity \eqref{p_o_t_d_r_qq_c.eq}. We first show a simpler variant of this identity, where we optimize over the class of convex functions $\psi$ in $\Cqs$ (i.e., which have quadratic growth), as in \eqref{prop_new_duality_first_eq_s.qq} of Proposition \ref{prop_new_duality_first.qq} below. 

\begin{proposition} \label{prop_new_duality_first.qq} The value $\tilde{D}^{q}(\mu,\nu)$ of the dual problem \eqref{eq_dual_qq} is equal to
\begin{equation} \label{prop_new_duality_first_eq_s.qq} 
D_{2}^{q}(\mu,\nu) \coloneqq 
\inf_{\substack{\psi \in \Cqs, \\ \textnormal{$\psi$ convex}}} 
\Big( \int \psi \, d\nu - \int (\psi^{\ast} \star q)^{\ast} \, d \mu \Big).
\end{equation}
\begin{proof} By Proposition \ref{prop_crucial_qq} we have
\[
\tilde{D}^{q}(\mu,\nu) = 
\inf_{\substack{\psi \in \Cqs, \\ \textnormal{$\psi$ convex}}}  
\Big( \int \psi \, d\nu - \int \varphi^{\psi} \, d \mu \Big).
\]
Hence it remains to show that $\varphi^{\psi} = (\psi^{\ast} \star q)^{\ast}$, for every convex function $\psi \in \Cqs$. In order to do this, we will first prove that $(\varphi^{\psi})^{\ast} = \psi^{\ast} \star q$. By definition of the convex conjugate \eqref{eq_def_conv_conj_qq} and of the maximization problem \eqref{eq_phi_psi_x_max.qq}, for $y \in \Rd$, we have
\begin{align}
(\varphi^{\psi})^{\ast}(y) 
&= \sup_{x \in \Rd} \big(\langle x,y \rangle - \varphi^{\psi}(x)\big) \label{prop_new_duality_first_d.qq_i} \\
&= \sup_{x \in \Rd} \sup_{p \in \PP_{2}^{x}(\Rd)} \Big( \MCov(p,q)-\int \psi_{y} \, dp\Big) \label{prop_new_duality_first_d.qq_ii} \\
&= \sup_{p \in \PP_{2}(\Rd)} \Big(\MCov(p,q)- \int \psi_{y} \, dp\Big), \label{prop_new_duality_first_d.qq_iii}
\end{align}
where the function $\psi_{y}$ is defined by $\psi_{y}(z) \coloneqq \psi(z) - \langle y,z \rangle$, for $z \in \Rd$. Now applying Lemma \ref{lem_phi_psi_gen.qq} to the proper convex function $\psi_{y}$ yields
\begin{equation} \label{prop_new_duality_first_d.qq}
(\varphi^{\psi})^{\ast}(y)  
= \phi^{\psi_{y}} 
= \int \psi_{y}^{\ast} \, dq 
= \int \psi^{\ast}(y+z)  \, dq(z) 
= (\psi^{\ast} \star q)(y),
\end{equation}
where for the last equality we recall the definition \eqref{eq_def_star_qq}. To complete the proof, we must justify that $(\varphi^{\psi})^{\ast\ast} = \varphi^{\psi}$, i.e., that the Fenchel--Moreau theorem is applicable. To see this, we first note that the function $x \mapsto \varphi^{\psi}(x)$ is convex and we have the upper bound $\varphi^{\psi} \leqslant \psi < + \infty$. Furthermore, for any $p \in \PP_{2}^{x}(\Rd)$ we have the inequalities
\[
\MCov(p,q) \leqslant 
\tfrac{1}{2} \int \vert y \vert^{2} \, dp(y) + 
\tfrac{1}{2} \int \vert z \vert^{2} \, dq(z) < + \infty
\]
as well as
\[
\int \psi \, dp \geqslant \ell + \tfrac{1}{2} \int \vert y \vert^{2} \, dp(y) > -\infty,
\]
the latter following from the fact that $\psi \in \Cqs$. As a consequence, we get the lower bound 
\[
\varphi^{\psi} \geqslant \ell - \tfrac{1}{2} \int \vert z \vert^{2} \, dq(z) > - \infty.
\]
Altogether, $\varphi^{\psi}$ is a convex function which is finite everywhere on $\R^{d}$, thus it is continuous and we indeed have $(\varphi^{\psi})^{\ast\ast} = \varphi^{\psi}$.
\end{proof}
\end{proposition}

Before we finally turn to the proof of Theorem \ref{theorem_new_duality_qq} at the end of this section, we formulate a ``relaxed'' version of \eqref{WeakDual_qq}, \eqref{prop_new_duality_first_eq_s.qq}. Similarly as in Lemma \ref{lem_ext_def_dual.qq} we now admit convex functions $\psi \colon \Rd \rightarrow (-\infty,+\infty]$ with $\mu(\dom \psi) = 1$ as dual variables, see \eqref{theorem_new_duality_sec_eq_s.qq} of Proposition \ref{prop_new_duality_rel.qq} below.

\begin{proposition} \label{prop_new_duality_rel.qq} The value $\tilde{D}^{q}(\mu,\nu)$ of the dual problem \eqref{eq_dual_qq} is equal to  
\begin{equation} \label{theorem_new_duality_sec_eq_s.qq}
D_{\textnormal{rel}}^{q}(\mu,\nu) \coloneqq \inf_{\substack{\mu(\dom \psi) = 1, \\ \textnormal{$\psi$ convex}}} \mathcal{E}(\psi),
\end{equation}
where the dual function $\psi \mapsto \mathcal{E}(\psi)$ is defined by
\begin{equation} \label{theorem_new_duality_sec_eq_s_rel_for.qq}
\mathcal{E}(\psi) \coloneqq \int \Big( \int \psi(y) \, \pi_{x}(dy) 
- (\psi^{\ast} \star q)^{\ast}(x) \Big) \, \mu(dx),
\end{equation}
with $\pi$ being an arbitrary fixed element of $\MT(\mu,\nu)$.
\begin{proof} We start with showing that the dual function $\mathcal{E}(\, \cdot \,)$ given in \eqref{theorem_new_duality_sec_eq_s_rel_for.qq} is well-defined for every convex function $\psi \colon \Rd \rightarrow (-\infty,+\infty]$ which is $\mu$-a.s.\ finite. To this end, we first prove the inequality $(\psi^{\ast} \star q)^{\ast} \leqslant \psi$. Since $\psi$ is convex with $\mu(\dom \psi) = 1$, for every $y \in \Rd$, the function 
\begin{equation} \label{eq.def.psiyz.qq}
\Rd \ni z \longmapsto \psi_{y}(z) \coloneqq \psi(z) - \langle y,z \rangle
\end{equation}
is a proper convex function, so that as in the proof of Proposition \ref{prop_new_duality_first.qq} above we can apply Lemma \ref{lem_phi_psi_gen.qq}. In particular, recalling the equation \eqref{prop_new_duality_first_d.qq}, we again have $\psi^{\ast} \star q = (\varphi^{\psi})^{\ast}$. Taking the convex conjugate and using that $\varphi^{\psi} \leqslant \psi$, we obtain 
\begin{equation} \label{eq.in.wd.star.psi.qq}
(\psi^{\ast} \star q)^{\ast} = (\varphi^{\psi})^{\ast\ast} \leqslant \varphi^{\psi} \leqslant \psi, 
\end{equation}
as required. Applying Jensen's inequality as in the proof of Lemma \ref{lem_ext_def_dual.qq}, this implies that the integrand in \eqref{theorem_new_duality_sec_eq_s_rel_for.qq} is $\mu$-a.s.\ non-negative, and hence $\mathcal{E}(\psi)$ is well-defined and takes values in the interval $[0,+\infty]$. 

\smallskip

Now let us turn to the proof of $\tilde{D}^{q}(\mu,\nu) = D_{\textnormal{rel}}^{q}(\mu,\nu)$. Recalling \eqref{theorem_new_duality_sec_eq_s.qq}, \eqref {prop_new_duality_first_eq_s.qq}, and using Proposition \ref{prop_new_duality_first.qq}, we have 
\[
D_{\textnormal{rel}}^{q}(\mu,\nu) \leqslant 
D_{2}^{q}(\mu,\nu) = 
\tilde{D}^{q}(\mu,\nu),
\]
so that we need to show the inequality $\tilde{D}^{q}(\mu,\nu) \leqslant D_{\textnormal{rel}}^{q}(\mu,\nu)$. To this end, let $\pi \in \MT(\mu,\nu)$ and $\psi \colon \Rd \rightarrow (-\infty,+\infty]$ be convex with $\mu(\dom \psi) = 1$. Since $\tilde{D}^{q}(\mu,\nu) = P^{q}(\mu,\nu)$ (recall Proposition \ref{prop_no_duality_gap_first_qq} and \eqref{eq_the_no_duality_gap_qq}), we have to verify the inequality
\[
\int \MCov(\pi_{x},q) \, d\mu(x)
\leqslant  \int \Big( \int \psi(y) \, \pi_{x}(dy)  - (\psi^{\ast} \star q)^{\ast}(x)\Big) \, \mu(dx).
\]
It is sufficient to prove, for $\mu$-a.e.\ $x \in \Rd$, that
\begin{equation} \label{cor_new_duality_sec_01.qq} 
\MCov(\pi_{x},q) 
\leqslant  \int \psi \, d\pi_{x} - (\psi^{\ast} \star q)^{\ast}(x).
\end{equation}
We express the convex conjugate on the right-hand side of \eqref{cor_new_duality_sec_01.qq} as
\[
-(\psi^{\ast} \star q)^{\ast}(x) 
= \inf_{y \in \Rd} \big( (\psi^{\ast} \star q)(y) - \langle x,y \rangle \big) 
= \inf_{y \in \Rd} \Big( \int \psi_{y}^{\ast} \, dq - \langle x,y \rangle \Big);
\]
for the last equality we again relied on the equation \eqref{prop_new_duality_first_d.qq}, and the function $\psi_{y}$ is defined as in \eqref{eq.def.psiyz.qq}. Substituting back into \eqref{cor_new_duality_sec_01.qq} yields
\begin{equation} \label{cor_new_duality_sec_02.qq} 
\MCov(\pi_{x},q) \leqslant \inf_{y \in \Rd} \Big(\int \psi_{y} \, d \pi_{x} + \int \psi_{y}^{\ast} \, dq\Big).
\end{equation}
This is what we need to show to complete the proof of Proposition \ref{prop_new_duality_rel.qq}. Now observe that by the Fenchel--Young inequality we have
\[
\MCov(\pi_{x},q) \leqslant \int f \, d \pi_{x} + \int f^{\ast} \, dq,
\]
for every proper convex function $f \colon \Rd \rightarrow (-\infty,+\infty]$. In particular, for every $y \in \Rd$, it holds that
\[
\MCov(\pi_{x},q) \leqslant \int \psi_{y} \, d \pi_{x} + \int \psi_{y}^{\ast} \, dq,
\]
which implies \eqref{cor_new_duality_sec_02.qq}, as required.
\end{proof}
\end{proposition}

Finally, we give the proof of Theorem \ref{theorem_new_duality_qq}, where we optimize over all $\nu$-integrable convex functions $\psi$ as in \eqref{WeakDual_qq}, \eqref{p_o_t_d_r_qq_c.eq}.

\begin{proof}[Proof of Theorem \ref{theorem_new_duality_qq}] We first show that the integral
\begin{equation} \label{int.ex.qq}
\int (\psi^{\ast} \star q)^{\ast} \, d \mu 
\end{equation}
appearing on the right-hand side of \eqref{WeakDual_qq}, \eqref{p_o_t_d_r_qq_c.eq} is well-defined for every convex function $\psi \in L^{1}(\nu)$. To see this, recall from \eqref{eq.in.wd.star.psi.qq} above that $(\psi^{\ast} \star q)^{\ast} \leqslant \psi$. On the other hand, since $\mu \lc \nu$ and $\psi \in L^{1}(\nu)$ is convex, the integral of $\psi$ with respect to $\mu$ exists. We conclude that also the integral \eqref{int.ex.qq} exists.

\smallskip

Now let us turn to the actual proof of Theorem \ref{theorem_new_duality_qq}. Recalling the beginning of Subsection \ref{p_o_t_d_r_qq_c}, we have to prove the equality \eqref{p_o_t_d_r_qq_c.eq}. But this assertion follows immediately from Proposition \ref{prop_new_duality_rel.qq} and Proposition \ref{prop_new_duality_first.qq} by a ``sandwich argument''. Indeed, we have
\begin{equation} \label{eq_sw_arg.qq}
\tilde{D}^{q}(\mu,\nu) = 
D_{\textnormal{rel}}^{q}(\mu,\nu) \leqslant 
D^{q}(\mu,\nu) \leqslant 
D_{2}^{q}(\mu,\nu) 
= \tilde{D}^{q}(\mu,\nu).
\end{equation}
The inequalities in \eqref{eq_sw_arg.qq} are due to the inclusions
\[
\{ \psi \in \Cqs\colon \textnormal{$\psi$ convex} \} \subseteq \{ \psi \in L^{1}(\nu) \colon \textnormal{$\psi$ convex} \} \subseteq \{ \textnormal{$\psi$ convex} \colon \mu(\dom \psi) = 1  \};
\]
the equalities on the left-hand side and on the right-hand side of \eqref{eq_sw_arg.qq} are justified by Proposition \ref{prop_new_duality_rel.qq} and Proposition \ref{prop_new_duality_first.qq}, respectively.

\smallskip

Finally, we recall that existence and uniqueness of the optimizer $\hat{\pi} \in \MT(\mu,\nu)$ as well as finiteness of $P^{q}(\mu,\nu)$ have already been shown in Proposition \ref{prop_no_duality_gap_first_qq}.
\end{proof}

\section{Existence and properties of \texorpdfstring{$q$}{q}-Bass martingales} \label{chap_qbms}

Throughout this chapter we fix $\mu, q \in \PP_{2}(\Rd)$ and assume that $q$ does not give mass to small sets. Our goal is to prove our second main result, Theorem \ref{qq_smr_qq}. Its proof is based on the work \cite{BVBST23}, where the Gaussian case $q = \gamma$ is studied in detail. First, we need some auxiliary results.

\begin{lemma} \label{lem_phipsi_con_qq} Let $\psi \colon \Rd \rightarrow (-\infty,+\infty]$ be a proper, lower semicontinuous convex function.
\begin{enumerate}[label=(\roman*)] 
\item \label{lem_phipsi_con_qq_i} The function $x \mapsto \varphi^{\psi}(x)$ defined in \eqref{eq_phi_psi_x_qq} is convex on $\dom\psi$.
\item \label{lem_phipsi_con_qq_ii} If $\varphi^{\psi}(x) > - \infty$, for some $x \in \operatorname{int} (\dom \psi)$, then $\varphi^{\psi} > - \infty$ on $\operatorname{int}(\dom \psi)$.
\end{enumerate}
\begin{proof} \ref{lem_phipsi_con_qq_i} To see that the function $x \mapsto \varphi^{\psi}(x)$ is convex on $\dom \psi$, we let $x \coloneqq \lambda x_{1}+(1-\lambda)x_{2}$, for $x_{1},x_{2} \in \dom \psi$ and $\lambda \in(0,1)$. If $\varphi^{\psi}(x_{1}) = -\infty$, then there are $p_{x_{1}}^{(n)} \in \PP_{2}^{x_{1}}(\Rd)$ with 
\[
\int \psi \, dp_{x_{1}}^{(n)} - \MCov(p_{x_{1}}^{(n)},q) \rightarrow -\infty.
\]
We observe that $p_{x}^{(n)} \coloneqq \lambda p_{x_{1}}^{(n)} + (1-\lambda)\delta_{x_{2}}\in  \PP_{2}^{x}(\Rd)$, and consequently
\begin{align*}
\varphi^{\psi}(x)
&\leqslant \lambda \int \psi \, dp_{x_{1}}^{(n)} + (1-\lambda) \psi(x_{2}) - \MCov(p_{x}^{(n)},q) \\
&\leqslant \lambda \Big( \int \psi \, dp_{x_{1}}^{(n)}  - \MCov(p_{x_{1}}^{(n)},q) \Big) 
+(1-\lambda)\psi(x_{2}),
\end{align*}
where we have used the convexity of $\PP_{2}(\Rd) \ni p \longmapsto - \MCov(p,q)$ in the last inequality. We conclude that  $\varphi^{\psi}(x) = -\infty$. The case $\varphi^{\psi}(x_{2}) = -\infty$ is treated similarly. If, on the other hand, both $\varphi^{\psi}(x_{1}) > - \infty$ and $\varphi^{\psi}(x_{2}) > -\infty$, then 
\[
\varphi^{\psi}(x)\leqslant \lambda \varphi^{\psi}(x_{1}) + (1-\lambda)\varphi^{\psi}(x_{2})
\]
follows by standard arguments. Altogether we see that $\varphi^{\psi}$ is convex on $\dom \psi$. 

\smallskip

\noindent \ref{lem_phipsi_con_qq_ii} If $\varphi^{\psi}(x) > - \infty$ for one $x \in \operatorname{int} (\dom \psi)$, then $\varphi^{\psi}(\tilde{x}) > - \infty$ for all $\tilde{x} \in \operatorname{int}(\dom \psi)$, as can be seen directly by convexity. 
\end{proof}
\end{lemma}

Our next result is the analogue of Lemma \ref{lem_phi_psi_gen.qq}. But now, instead of maximizing over all $p \in \PP_{2}(\Rd)$ as in \eqref{eq_phi_psi.qq}, we maximize over all $p \in \PP_{2}^{x}(\Rd)$, with a fixed barycenter $x \in \Rd$, as in \eqref{eq_phi_psi_x_max.qq}.

\begin{lemma} \label{lem_phi_psi_x_gen_qq} Let $\psi \colon \Rd \rightarrow (-\infty,+\infty]$ be a lower semicontinuous convex function and assume that $\varphi^{\psi}(x) > - \infty$ for some $x \in \operatorname{int} (\dom \psi)$. Then we have the duality formula
\begin{equation} \label{eq_prop_phi_psi_x_gen_df.qq}
\varphi^{\psi}(x) = \sup_{y \in \Rd} \Big(\langle x,y \rangle - \int \psi^{\ast}(y+z)  \, dq(z)\Big).
\end{equation}
\begin{proof} We first show that $(\varphi^{\psi})^{\ast} = \psi^{\ast} \star q$. To see this, we recall equations \eqref{prop_new_duality_first_d.qq_i} -- \eqref{prop_new_duality_first_d.qq} of Proposition \ref{prop_new_duality_first.qq} and note that these are valid for any proper, lower semicontinuous convex function $\psi$. Therefore, for each $y \in \Rd$, we have
\begin{equation} \label{eq_phi_psi_x_gen_01_i_qq_a}
(\varphi^{\psi})^{\ast}(y) 
= \sup_{p \in \PP_{2}(\Rd)} \Big(\MCov(p,q)- \int \psi_{y} \, dp\Big)
= \int \psi^{\ast}(y+z)  \, dq(z),
\end{equation}
where the function $\psi_{y}$ is defined by $\psi_{y}(z) \coloneqq \psi(z) - \langle y,z \rangle$, for $z\in \Rd$. 

\smallskip

Since $\varphi^{\psi}(x) > - \infty$ for some $x \in \operatorname{int} (\dom \psi)$ and $\varphi^{\psi}$ is convex by Lemma \ref{lem_phipsi_con_qq}, the function $\varphi^{\psi}$ is finite and thus also continuous in a neighbourhood of $x$. Hence we can apply the variant \cite[Proposition 13.44]{BC17} of the Fenchel--Moreau theorem and deduce from \eqref{eq_phi_psi_x_gen_01_i_qq_a} that
\begin{equation} \label{eq_phi_psi_x_gen_01_ii.qq}
\varphi^{\psi}(x) 
= (\varphi^{\psi})^{\ast\ast}(x) 
= \sup_{y \in \Rd} \Big(\langle x,y \rangle - \int \psi^{\ast}(y+z)  \, dq(z)\Big),
\end{equation}
which proves the duality formula \eqref{eq_prop_phi_psi_x_gen_df.qq}.
\end{proof}
\end{lemma}

The duality formula \eqref{eq_prop_phi_psi_x_gen_df.qq} established in Lemma \ref{lem_phi_psi_x_gen_qq} enables us to generalize Lemma \ref{lem_phi_psi.qq} and identify the structure of the optimizer in the minimization problem \eqref{eq_phi_psi_x_qq}.

\begin{lemma} \label{lem_phi_psi_x_gen_opt_qq} Let $\psi \colon \Rd \rightarrow (-\infty,+\infty]$ be a lower semicontinuous convex function.
\begin{enumerate}[label=(\roman*)] 
\item \label{lem_phi_psi_x_gen_opt_qq_i} If there are $x \in \Rd$ and $\hat{y}(x) \in \Rd$ such that
\begin{equation} \label{lem_phi_psi_x_gen_opt_qq_i_prob}
\hat{p}_{x} \coloneqq \nabla \psi^{\ast}\big(\hat{y}(x) + \, \cdot \, \big)(q) \in \PP_{2}^{x}(\Rd),
\end{equation}
then this is actually the unique optimizer of the infimum in \eqref{eq_phi_psi_x_qq}. 
\item \label{lem_phi_psi_x_gen_opt_qq_ii} If additionally $\varphi^{\psi}(x) > - \infty$ and $x \in \operatorname{int} (\dom \psi)$, then $\hat{y}(x)$ is the optimizer of the supremum in \eqref{eq_prop_phi_psi_x_gen_df.qq}.
\end{enumerate}
\begin{proof} \ref{lem_phi_psi_x_gen_opt_qq_i} For fixed $y \in \Rd$ we define the function $\psi_{y}$ by $\psi_{y}(z) \coloneqq \psi(z) - \langle y,z \rangle$, for $z \in \Rd$, so that $\psi_{y}^{\ast}(z) = \psi^{\ast}(y+z)$. With this notation, we can express the probability measure defined in \eqref{lem_phi_psi_x_gen_opt_qq_i_prob} as $\hat{p}_{x} = (\nabla \psi_{\hat{y}(x)}^{\ast})(q)$. We denote by $T_{q}^{p_{x}}$ the Brenier map from $q$ to $p_{x} \in \PP_{2}^{x}(\Rd)$ and note that $T_{q}^{\hat{p}_{x}} = \nabla \psi_{\hat{y}(x)}^{\ast}$ by Brenier's theorem. For arbitrary $p_{x} \in \PP_{2}^{x}(\Rd)$ we compute
\begin{align*}
\int \psi \, dp_{x} - \MCov(p_{x},q)
&= \int \Big( \psi_{\hat{y}(x)}\big(T_{q}^{p_{x}}(z)\big)
+ \big\langle T_{q}^{p_{x}}(z), \hat{y}(x) \big\rangle
- \big\langle T_{q}^{p_{x}}(z),z\big\rangle 
 \Big) \, dq(z) \\
&\geqslant 
\big\langle x, \hat{y}(x) \big\rangle 
+ \int \inf_{y \in \Rd} \big( 
\psi_{\hat{y}(x)}(y) - \langle y,z\rangle \big) \, dq(z)  \\
&= \big\langle x, \hat{y}(x) \big\rangle 
- \int \psi^{\ast}(\hat{y}(x) + z) \, dq(z)  \\
&= \int \psi \, d\hat{p}_{x} - \MCov(\hat{p}_{x},q),
\end{align*}
with equality if and only if $T_{q}^{p_{x}}(z) = T_{q}^{\hat{p}_{x}}(z)$, for $q$-a.e.\ $z \in \Rd$. This in turn is the case if and only if $p_{x} = \hat{p}_{x}$. We conclude that $\hat{p}_{x}$ is the unique optimizer of the infimum in \eqref{eq_phi_psi_x_qq}, i.e.,
\begin{equation} \label{eq_phi_psi_x_gen_opt_i_qq}
\varphi^{\psi}(x) 
= \int \psi \, d\hat{p}_{x} - \MCov(\hat{p}_{x},q) 
= \big\langle x,\hat{y}(x) \big\rangle - \int \psi^{\ast}\big(\hat{y}(x) + z\big) \, dq(z).
\end{equation}

\smallskip

\noindent \ref{lem_phi_psi_x_gen_opt_qq_ii} If $\varphi^{\psi}(x) > - \infty$ and $x \in \operatorname{int} (\dom \psi)$, then by Lemma \ref{lem_phi_psi_x_gen_qq} we have that
\[
\varphi^{\psi}(x) = \sup_{y \in \Rd} \Big(\langle x,y \rangle - \int \psi^{\ast}(y+z)  \, dq(z)\Big),
\]
which in light of \eqref{eq_phi_psi_x_gen_opt_i_qq} implies that $\hat{y}(x)$ is the optimizer of the supremum in \eqref{eq_prop_phi_psi_x_gen_df.qq}.
\end{proof}
\end{lemma}

\subsection{Proof of Theorem \texorpdfstring{\ref{qq_smr_qq}}{1.5}} \label{qq_smr_qq_sec_15}

\begin{definition} \label{def_dual_opt_qq} We say that a lower semicontinuous convex function $\hat{\psi} \colon \Rd \rightarrow (-\infty,+\infty]$ satisfying $\mu(\operatorname{int}(\dom\hat{\psi}))=1$ is an optimizer of the dual problem \eqref{WeakDual_qq} if $D^{q}(\mu,\nu) = \mathcal{E}(\hat{\psi})$, for the dual function $\mathcal{E}(\, \cdot \,)$ as defined in \eqref{theorem_new_duality_sec_eq_s_rel_for.qq}. In short, we say that $\hat{\psi}$ is a dual optimizer.
\end{definition}

\begin{proof}[Proof of Theorem \ref{qq_smr_qq}] \ref{qq_smr_qq_i} Let $\hat{v} \colon \Rd \rightarrow \R$ be $q$-Bass martingale generating with initial marginal $\mu$. Then, according to Definition \ref{def_psi_v_qbg}, the function $\hat{v}$ is convex, its convex conjugate $\hat{\psi} = \hat{v}^{\ast}$ satisfies $\mu(\operatorname{int}(\dom\hat{\psi})) = 1$, and the finite-valued function $\hat{v} \star q$ is strictly convex with gradient 
\begin{equation} \label{qq_smr_qq_i_bwi_grad}
\nabla(\hat{v} \star q) = (\nabla \hat{v}) \star q.
\end{equation}
By \cite[Theorem 26.5]{Roc70}, it follows that the convex conjugate $(\hat{v} \star q)^{\ast}$ is differentiable and strictly convex on $\operatorname{int}(\dom(\hat{v} \star q)^{\ast})$. Furthermore, the function 
\[
\nabla(\hat{v} \star q) \colon \Rd \longrightarrow \operatorname{int} (\dom (\hat{v} \star q)^{\ast})
\]
is a bijection with inverse
\begin{equation} \label{qq_smr_qq_i_bwi}
(\nabla \hat{v} \star q)^{-1} = \nabla(\hat{v} \star q)^{\ast}.
\end{equation}

\smallskip

Next, we show that $\mu(\operatorname{int} (\dom (\hat{v} \star q)^{\ast})) = 1$. To this end, we will prove the inclusion 
\[
\operatorname{int} (\dom \hat{\psi}) \subseteq \operatorname{int} (\dom (\hat{v} \star q)^{\ast}). 
\]
Indeed, for $x \in \dom \hat{\psi}$, we have
\begin{align*}
(\hat{v} \star q)^{\ast}(x)
&= \sup_{y \in \Rd} \Big( \langle x,y \rangle - \int \hat{v}(y+z) \, dq(z)\Big) 
\leqslant \int \sup_{y \in \Rd} \Big( \langle x,y \rangle -\hat{v}(y+z) \Big) \, dq(z) \\
&= \int \sup_{y \in \Rd} \Big( \langle x,y-z \rangle - \hat{v}(y) \Big) \, dq(z) 
= \int \Big( \hat{v}^{\ast}(x) - \langle x,z \rangle \Big) \, dq(z) \\
&\leqslant \hat{v}^{\ast}(x) + \vert x \vert \
\sqrt{\int \vert z \vert^{2} \, d q(z)} < + \infty,
\end{align*}
so that $\dom \hat{\psi} \subseteq \dom (\hat{v} \star q)^{\ast}$.
 
\smallskip

Now we are ready to show that the pair $(\hat{v},\hat{\alpha}^{\hat{v}})$ is a $q$-Bass martingale with initial marginal $\mu$ and with terminal marginal $\nu^{\hat{v}}$, i.e.,
\begin{equation} \label{qq_smr_qq_i_dpqbm}
(\nabla \hat{v} \star q)(\hat{\alpha}^{\hat{v}}) = \mu
\qquad \textnormal{ and } \qquad 
\nabla \hat{v}(\hat{\alpha}^{\hat{v}} \ast q) = \nu^{\hat{v}}.
\end{equation}
According to Definition \ref{def_psi_v_qbg}, we have $\hat{\alpha}^{\hat{v}} = \nabla(\hat{v} \star q)^{\ast}(\mu)$ and $\nu^{\hat{v}} = \nabla \hat{v}(\hat{\alpha}^{\hat{v}} \ast q)$. Together with \eqref{qq_smr_qq_i_bwi}, we conclude the relations \eqref{qq_smr_qq_i_dpqbm}.

\medskip

\noindent \ref{qq_smr_qq_ii} Recalling the second equality in \eqref{eq_phi_psi_x_gen_01_ii.qq}, we have
\begin{equation} \label{eq_du_op_b_m_psi_v_a_02.qq}
\varphi^{\hat{\psi}}(x) \geqslant 
(\varphi^{\hat{\psi}})^{\ast\ast}(x) 
= \sup_{y \in \Rd} \Big(\langle x,y \rangle - \int \hat{v}(y+z)  \, dq(z)\Big).
\end{equation}
As the right-hand side of \eqref{eq_du_op_b_m_psi_v_a_02.qq} is greater than $-\infty$, for every $x \in \Rd$, we can apply Lemma \ref{lem_phi_psi_x_gen_qq}. Consequently, for every $x \in \operatorname{int} (\dom \hat{\psi})$, we obtain an equality in \eqref{eq_du_op_b_m_psi_v_a_02.qq}, i.e.,
\begin{equation} \label{eq_du_op_b_m_psi_v_a_02.qq_i}
\varphi^{\hat{\psi}}(x) 
= \sup_{y \in \Rd} \Big(\langle x,y \rangle - (\hat{v} \star q)(y) \Big).
\end{equation}
Recall that the finite-valued function $\hat{v} \star q$ is strictly convex and differentiable with gradient \eqref{qq_smr_qq_i_bwi_grad}. The first order condition for an optimizer $\hat{y}(x) \in \Rd$ of the right-hand side of \eqref{eq_du_op_b_m_psi_v_a_02.qq_i} reads
\begin{equation} \label{eq_du_op_b_m_psi_v_a_03.qq}
(\nabla \hat{v} \star q)\big(\hat{y}(x)\big) = x.
\end{equation}
From \eqref{qq_smr_qq_i_bwi} we deduce that the right-hand side of \eqref{eq_du_op_b_m_psi_v_a_02.qq_i} admits a unique maximizer $\hat{y}(x)$ given by
\[
\hat{y}(x) 
= (\nabla \hat{v} \star q)^{-1}(x)
= \nabla(\hat{v} \star q)^{\ast}(x),
\]
for all $x \in \operatorname{int} (\dom \hat{\psi})$. 

\smallskip

For each $x \in \operatorname{int} (\dom \hat{\psi})$, we define a probability measure 
\[
\pi_{x}^{\hat{v}} \coloneqq \nabla \hat{v}\big(\hat{y}(x) + \, \cdot \,\big)(q).
\]
Since $\nu^{\hat{v}}$ has finite second moment, we observe that
\begin{equation} \label{eq_du_op_b_m_psi_v_a_00.qq}
\int \vert \nabla \hat{v}(y + z) \vert^{2} \, dq(z) < + \infty,
\end{equation}
for $\hat{\alpha}^{\hat{v}}$-a.e.\ $y \in \Rd$. By \eqref{eq_du_op_b_m_psi_v_a_03.qq} and \eqref{eq_du_op_b_m_psi_v_a_00.qq}, $\pi_{x}^{\hat{v}}$ is an element of $\PP_{2}^{x}(\Rd)$, for $\mu$-a.e.\ $x \in \Rd$. Hence we can apply part \ref{lem_phi_psi_x_gen_opt_qq_i} of Lemma \ref{lem_phi_psi_x_gen_opt_qq}, showing that the infimum
\[
\varphi^{\hat{\psi}}(x) 
= \inf_{p \in \PP_{2}^{x}(\Rd)} \Big( \int \hat{\psi} \, dp - \MCov(p,q)\Big)
\]
is attained by $\pi_{x}^{\hat{v}}$, for $\mu$-a.e.\ $x \in \Rd$. Therefore, from the representation \eqref{eq_rep_dual_func_qq} of the dual function \eqref{eq_def_dual_func_qq} and the fact that $\pi_{x}^{\hat{v}}(dy) \, \mu(dx) \in \MT(\mu,\nu^{\hat{v}})$, we obtain
\begin{equation} \label{eq_du_op_b_m_psi_v_a_04.qq}
\tilde{\mathcal{E}}(\hat{\psi}) = \int  \Big( \int \hat{\psi}(y) \, \pi^{\hat{v}}_{x}(dy)  - \varphi^{\hat{\psi}}(x) \Big) \, \mu(dx)
= \int \MCov(\pi_{x}^{\hat{v}},q) \, \mu(dx).
\end{equation}
On the other hand, from \eqref{eq_primal_qq}, Proposition \ref{prop_no_duality_gap_first_qq} and \eqref{lem_ext_dual_func_01.qq} we have
\begin{equation} \label{eq_du_op_b_m_psi_v_a_04.qq_i}
\int \MCov(\pi_{x}^{\hat{v}},q) \, \mu(dx) 
\leqslant P^{q}(\mu,\nu^{\hat{v}})
= \tilde{D}^{q}(\mu,\nu^{\hat{v}}) 
\leqslant \tilde{\mathcal{E}}(\hat{\psi}),
\end{equation}
which in light of \eqref{eq_du_op_b_m_psi_v_a_04.qq} implies that
\[
P^{q}(\mu,\nu^{\hat{v}}) = \int \MCov(\pi_{x}^{\hat{v}},q) \, \mu(dx).
\]
By uniqueness of the optimizer $\hat{\pi} \in \MT(\mu,\nu^{\hat{v}})$, we conclude that $\pi_{x}^{\hat{v}} = \hat{\pi}_{x}$, for $\mu$-a.e.\ $x \in \Rd$.

\medskip

\noindent \ref{qq_smr_qq_iii} According to \eqref{eq_du_op_b_m_psi_v_a_02.qq_i}, for every $x \in \operatorname{int} (\dom \hat{\psi})$, we have that
\[
\varphi^{\hat{\psi}}(x) 
= (\hat{v} \star q)^{\ast}(x).
\]
Recalling the definition of the dual function \eqref{theorem_new_duality_sec_eq_s_rel_for.qq}, we observe that
\begin{align*}
\mathcal{E}(\hat{\psi})
&= \int \Big( \int \hat{\psi}(y) \, \pi_{x}(dy) 
- (\hat{\psi}^{\ast} \star q)^{\ast}(x) \Big) \, \mu(dx) \\
&= \int  \Big( \int \hat{\psi}(y) \, \pi^{\hat{v}}_{x}(dy)  - \varphi^{\hat{\psi}}(x) \Big) \, \mu(dx)
= \tilde{\mathcal{E}}(\hat{\psi}).
\end{align*}
Finally, from \eqref{eq_du_op_b_m_psi_v_a_04.qq} and \eqref{eq_du_op_b_m_psi_v_a_04.qq_i} we deduce that $\hat{\psi}$ is a dual optimizer in the sense of Definition \ref{def_dual_opt_qq}. 
\end{proof}

\bibliographystyle{alpha}
{\footnotesize
\bibliography{references}}

\end{document}